\documentclass[12pt]{article}
\usepackage[utf8]{inputenc}
\usepackage{array}
\usepackage[english,activeacute]{babel}
\usepackage{amsmath,csquotes,amsfonts,amssymb,amstext,amsthm,amscd,mathrsfs,amsbsy,graphics,graphicx}
\usepackage{xypic}
\usepackage[all]{xy}
\newtheorem{theo}{Theorem}[section]
\newtheorem{pro}[theo]{Proposition}
\newtheorem{coro}[theo]{Corollary}

\newtheorem{con}[theo]{Conjecture}
\theoremstyle{definition}
\newtheorem{defi}[theo]{Definition}

\newtheorem{rem}[theo]{Remark}

\begin{document}
\title{A proof of the Generalized Jacobian conjecture}
\author{\\Quan XU \footnote {The author is supported by start-up funds of No.190738 in College of Sciences, China jiliang University.} 
}
\maketitle. 

\begin{abstract}
Based on the reduction of degree in polynomial mappings and some known results in algebraic geometry, by introducing the Brouwer degree, a tool from differential topology, algebraic topology and algebraic geometry, we completely prove the Generalized Jacobian conjecture in the field of real numbers, which implies the Generalized complex Jacobian conjecture. Also, for the strong real Jacobian conjecture, we present a newly sufficient and necessary condition.
\end{abstract}

\section{Introduction}
Notation:\\
$\mathbb{C}$: the field of complex numbers; $\mathbb{R}$: the field of real numbers;\\
 Let $\mathbb{K}$ be a field either $\mathbb{C}$ or $\mathbb{R}$.
A mapping $F :\mathbb{K}^{n}\rightarrow \mathbb{K}^{n}$ is called the polynomial mapping if $F=(F_{1},\cdots , F_{n})$, with $F_{i}\in \mathbb{K}[x_{1},\cdots , x_{n}]$ for $i=1,\cdots, n$.\\
$JF(x)$: the Jacobian matrix of $F(x)$; $\det JF(x)$ : the determiant of $JF(x)$.\\
$Im(F)$: the set of images of the mapping $F$.

In 1939, Ott-Heinrich Keller proposed the following question in \cite{kel}: Given polynomials $F_{1},\cdots, F_{n}\in \mathbb{Z}[x_{1},\cdots,x_{n}]$ such that $\det J(F)=1$, where  $J(F)$ denotes the Jacobian matrix $ (\frac{\partial F_{i}}{\partial x_{j}})_{(i,j)}$, can every $x_{i}$ be expressed as a polynomial in $F_{1},\cdots, F_{n}$ with coefficients in $\mathbb{Z}$ ?   Keller's original question now is known as the famous  Jacobian Conjecture:
\begin{con}{\bf (Jacobian Conjecture over $\mathbb{C}$)}
 Let $ F : \mathbb{C}^{n} \to \mathbb{C}^{n} $ be a polynomial mapping. Then $F$ is a polynomial automorphism if and only if the determinant of its Jacobian of the polynomial mapping $F$
 is a non-zero constant, i.e., $ \det JF(x)= c \in \mathbb{C}- \{0\}, \text{for}~\forall x\in \mathbb{C}^{n}$, where the polynomial automorphism means that the inverse mapping of $F$ exists and its inverse mapping is a polynomial mapping again.
  \end{con}
The necessity of the conjecture is trivial by the chain rule of composite functions. The sufficiency is just the generalization of Keller's orginal question. To the sufficiency, we have following theorem:
 \begin{theo}(\cite{BCW},\cite{BiBIRO})
Any injective polynomial mapping $F: \mathbb{C}^{n}\rightarrow \mathbb{C}^{n}$ is a polynomial automorphism.
\end {theo}
Under the theorem above, the formulation of Jacobian conjecture over $\mathbb{C}$ turns out to be that a polynomial mapping $F: \mathbb{C}^{n}\rightarrow \mathbb{C}^{n}$ with $\det JF(x)= \text{constant}\neq 0$, then $F$ is injective. Due to the Lefschetz principle, one can verify that the Jacobian conjecture
over $\mathbb{C}$ covers the case of Jacobian conjecture formulated for any field (or domain ) of charateristic zero. Hence, let $n\geq 2 $ and the $n$-dimensional Jacobian conjecture (for short $(JC)_{n}$) is formualted as follwing: Let $F$ be a polynomial mapping $F :\mathbb{K}^{n}\rightarrow \mathbb{K}^{n}$. 
  $$ (JC)_{n}~~~~~~~  [\det JF(x)=\text{constant}\neq 0] \Rightarrow  [F~\text {is injective}].$$
Also, the Generalized Jacobian conjecture ( for short (GJC)) is 
   $$ (GJC)~~~~ (JC)_{n}~\text{holds for every}~n\geq 2.$$
  It is worthy to note that if $\mathbb{K}$ is the field $\mathbb{R}$ of real numbers, any injective polynomial mapping $F$ is bijective (\cite{BiBIRO},\cite{CR}), but its inverse is not necessarily a polynomial mapping. For example, the polynomial mappping $F(x)=x^{3}+x $ is a injective from $\mathbb{R}$ to $\mathbb{R}$, but its inverse is not a polynomial mapping any more.

The Jacobian conjecture in $\mathbb{C}$ is not completely equivalent to the Jacobian conjecture in $\mathbb{R}$, but the following facts show 
the connection between these conjectures.
Let $F:\mathbb{C}^{n} \rightarrow \mathbb{C}^{n}$ be a complex polynomial mapping: $z=(z_{1},\cdots,z_{n})\mapsto (F_{1},\cdots,F_{n})$ and consider the associated real polynomial mapping $\widetilde{F}:\mathbb{R}^{2n}\to \mathbb{R}^{2n}$ which sends \\
$~~~~~~~(x_{1},y_{1},x_{2},y_{2},\cdots,x_{n},y_{n})~\text{to} ~(Re F_{1},Im F_{1}, \cdots, Re F_{n}, Im F_{n})$,
 where $z_{k}=x_{k}+iy_{k}, Re F_{i}$ and $Im F_{i}$ are the real part and the imaginary part of $F_{i}$, respectively. By the Laplace theorem,  it is known that $\det J\widetilde{F}=|\det JF(z)|^{2}$ which implies that $\det JF(z)$ is a non-zero constant if only if $\det J\widetilde{F}$ is a non-zero constant. Also, $F$ is injective if and only if $\widetilde{F}$ is injective. Hence it is evident that 
  $$ (JC)_{2n}~~~~\text{for}~~ \mathbb{R}[x_{1},\cdots,x_{2n}]\Rightarrow (JC)_{n}~~ \text{for}~~ \mathbb{C}[x_{1},\cdots,x_{n}],$$
i.e., the real (GJC) implies complex (GJC). However, it is still unclear whether the real Jacobian conjecrue $(JC)_{n}$ implies the complex Jacobian conjecture $(JC)_{n}$.

Note that for the complex polynomial mapping $F$, if  $\det JF(x)$ is not zero  everywhere for $x\in \mathbb{C}^{n}$, then  $\det JF(x)$ must be a non-zero constant by the algebraic closedness of $\mathbb{C}$. Once, it was asked whether $F$ is a injective for real polynomial mapping with $\det JF(x)>0$ (or $<0$) everywhere for $ x\in \mathbb{R}^{n}$, which is the so-called strong real Jacobian conjecture:
\begin{con}(Strong real Jacobian conjeture )
If $F: \mathbb{R}^n\rightarrow \mathbb{R}^n$ is a polynomial mapping and $\det J(F)(x)$ is not zero everywhere in $\mathbb{R}^n$, then $F$ is an injective mapping.
\end{con}
It is a pity that the conjecture is false and Pinchuk (\cite{Pin}) constructed a counterexample for $n=2$. In the end of this article, we will give a sufficient and necessary condition such that it is true.

The paper is constructed by four sections. In the second part, we will introduce some recent progress on the conjecure and show its connections with other mathematical field, even thought some results may be not found because of my ignorance. The third part is an introduction of  the main tool, i.e., the Brouwer degree, whose homotopical invariance is the only property employed in this proof of the Generalized Jacobian conjecture over $\mathbb{R}$. The fourt part gives the proof by known theorems and the Brouwer degree. Also, from our proof, we can obtain a newly sufficnent and necessary condition to the strong real Jacobian conjecture. Our main results are the Thm.\ref{injct} and Thm.\ref{srp}
  
\section{Known results and the Jacobian conjecture in other subjects}
There are too many experts who already made contributions to the Jacobian conjecture. In \cite{Esse}, A. Van dan den Essen already introduced the conjecture from many aspects and pointed out amount of connections with other fields in mathematic. Also, a lot of good references are listed, which is very wonderful resource to researchers. Still, I will continue to mention several topics related to the conjecture and some recent progress for integrity of logic.

  $\bf{(2.1)}$The Jacobian conjecture is famous in  algebraic geometry because of Abhyankar's work on the
 formal inversion formula (see\cite{Abhy}). He can constructed a formal inversion, i.e., an formal power series by using
 differential operators for a polynomial map. The inversion formula was first discovered by Guajar (unpublished). Their formula now is called Abhyankar-Gurjar formula which is simplified by Bass , Connell and Wright (see \cite{BCW}). Since the formula is from the utilization of differential operators, so the method is related to $D$-modules (see page 263,\cite{Esse}).
 
$\bf{(2.2)}$ By Bass, Connell, and Wright  in \cite{BCW}, they proved the following theorem: If we consider the Generalzied Jacobian conjecture (GJC) over $\mathbb{K}$, it suffices to consider for all $n\geqslant 2$ and all polynomial map of the form $I+H$ where $I$ is the identity and $H$ is a cubic homogeneous.
 Furthermore,  Dru$\dot{z}$kowski in \cite{DruTu1} proved that it is sufficient to prove  the Generalized Jacobian conjecture (over $\mathbb{K}$ ) only if we  consider all  special polynomial  maps of form $F=I+H=(x_{1}+H_{1},\cdots, x_{n}+H_{n})$  with $H_{i}=(\sum a_{ij}x_{j})^{3}, i=1,\cdots,n$ for every $n\geq 2$. 
  
$\bf{(2.3)}$ From the topological point of view, Gutierrez and Maquera in \cite{GuMa} proved that if $F : \mathbb{R}^{3}\to\mathbb{R}^{3}$ is a polynomial map  with $\det J(F) \neq 0$ everywhere in $\mathbb{R}^{3}$ such that
 $\text{Spec}(F)\bigcap [0, \epsilon )=\emptyset $, for some $\epsilon > 0$, and $\text{codim}(S_{F})\geq 2$ 
where $\text{Spec}( F)$ is a set of eigenvalue of the Jacobian of $F$ and $S_{F}$ is the set of points on which $F$ is not proper, then $F$ is bijective. In \cite{FerMaqVen},  they have obtained somehow general results by using the semi-algebraic maps instead of polynomial maps.

$\bf{(2.4)}$ The equivalence of the Jacobian conjecture, the Diximier conjecture and Poisson conjecture.

Briefly,we make some introduction for the two alien conjectures.  Let $R$ be a commutative ring with identity $1$ and $n$ a positive integer. The polynomial ring over $R$ in $n$-variables $x_{1},\cdots ,x_{n}$ is denoted by $ R[x_{1},\cdots ,x_{n}]$. The $n$-th Weyl algebra over $R$, denoted by $A_{n}(R)$, is the associctive $R$-algebra with generators $y_{1},\cdots,y_{2n}$ and relations
$[y_{i},y_{i+n}]=1$ for all $1\leq i\leq n$ and $[y_{i},y_{j}]=0$ otherwise ,where $[,]$ is the lie bracket. Dixmier Coljecture chaims:  for $n\geq 1$, every endomorphism of $A_{n}(\mathbb{C})$ is an automorphism.

The $n$-th Poisson algebra $P_{n}(R)$ over $R$ is the polynomial ring $R[x_{1},\cdots ,x_{2n}]$
 endowed with the canoical Poisson bracket $\{,\}$ defined by
  $$\{f,g\}=\sum_{i=1}^{n}(\frac{\partial f}{\partial x_{i}}\frac{\partial g}{\partial x_{i+n}}-\frac{\partial f}{\partial x_{i+n}}\frac{\partial g}
{\partial x_{i}}).$$
A  $\varphi$ endomorphism of $R[x_{1},\cdots,x_{2n}]$ is called an endomorphism of $P_{n}(R)$ if $\varphi$ preserves the Poisson bracket $\{ ,\}$
i.e.,$\varphi \{f, g\} =\{\varphi (f), \varphi (g)\}$ for all $f, g$ in $R[x_{1},\cdots,x_{2n}]$.
 Poisson conjecture claims:  Let $n\geq 1$ , every endomorphism of $P_{n}(\mathbb{C})$ is an automorphism.
We know that Dixmier conjecture implies the Jacobian conjecture over $\mathbb{C}$ (see \cite{Esse} and \cite{BCW}). Also, Tsuchimoto (see\cite{Tsu}) have proved that conversely, $(JC)_{2n}$ implies the $n$-dimensional Dixmier conjecture by $p$-curvature method.
Independently,  Belov and Kontsevich proved that the Jacobian conjecture is stably equivalent to the Dixmier conjecture (see \cite{BeKon}), whose proof is displayed throughoutly by the language of algebraic geometry. Inspired by the work of \cite{BeKon}, Essen by drawing Poisson conjecture into the Jacobian conjecture and the Dixmier conjecture, prove that all three conjectures are equivalent, whose method in the proof is purely algebraic.

\section{Brief introduction to the Brouwer degree}
In this section, we introduce the Brouwer degree. Firstly, we give some notations and conventions. Let $\mathbb{R}^{n}$ be $n$-dimensional space of real vectors $x=(x_{1},\cdots ,x_{n})$ with norm $|x|=(x_{1}^{2}+\cdots +x_{n}^{2})^{\frac{1}{2}}$. The boundary and the closure of an open set $D$ in $\mathbb{R}^{n}$ will be denoted by $\partial {D}$  and $\bar{D}=D\cup\partial{D}$, respectively. A mapping $F=F(x)$ of $D$ into $F(D)\subset \mathbb{R}^{n}$ is said to be of class $C^{k}(k\geq 0)$ in $D$ if it can be represented in the form $F_{i}=F_{i}(x_{1},\cdots, x_{n})~~~(i=1,\cdots ,n)$, where all component $ F_{i}(x_{1},\cdots, x_{n})$ belongs to $C^{k}$ in $D$.
The following definition was made by Heinz in \cite{Heinz}.
\begin{defi}
Let the mapping $F=F(x)$ be of the class $C^{1}$ in a bounded open set $D\subset \mathbb{R}^{n}$ and continuous in $\partial D$. Furthermore, let 
$F(x)\neq z$ for $x\in \partial D$, where $z$ is fixed in $\mathbb{R}^{n}$, and let a real-valued function $\Phi(r)$ be chosen such that 
the following conditions are satisfied:

(i)  $\Phi(r)$ is continuous in the interval $0\leq r < \infty$. furthermore, it vanishes in a neighbourhood of $r=0$ and 
for $\epsilon \leq r < \infty$, where $0<\epsilon < \text{min}_{x\in \partial D}|F(x)-z|$.

(ii) We have $\int_{\mathbb{R}^{n}}\Phi(|x|)dx =1.$ \\
Then Brouwer degree $d[F(x); D,z]$ is uniquely defined by the equation 
                        $$d[F(x); D,z]=\int_{D}\Phi(|F(x)-z|)\det (J[F(x)])dx$$
\end{defi}
In fact, we need to justify the definition , because the function $\Phi(x)$ which satifies condition (i) and (ii) above is probably not unique. 
The justification of the definition is made in Page 233-234 of \cite{Heinz}.

Moreover, the Brouwer degree can be defined for any continuous mappings, which is accomplished in Definition 2 of \cite{Heinz}.We directly show the defintion as follows:
\begin{defi}
Let the mapping $y=F(x)$ be continuous for $x\in \bar{D}$ and $F(x)\neq z$ for $x \in\partial D$ , where $D$ is a bounded open 
set in $\mathbb{R}^{n}$ and $z$ is fixed. Furtheremore, let $\{y_{k}(x)\}(k=1,2,\cdots)$ be a sequence of mappings which are of class
$C^{1}$ for $x\in \mathbb{R}^{n}$ and satisfy the relations $$y_{k}(x)\neq z~~~~~~(x\in \partial D)$$ 
and   $$\lim_{k\rightarrow \infty}y_{k}(x)=F(x)~~~~(x\in\bar{D} ),$$ 
where the convergence is uniform on $\bar{D}$. Then the Brouwer degree $d[F(x); D, z]$ is uniquely defined by the equation
      $$d[F(x); D, z]=\lim_{k\rightarrow \infty}d[y_{k}(x); D, z].$$
\end{defi}
The key point for the definiton is the existence of sequence of mappings, which does exist (see Definition 2 in \cite{Heinz} for more details).
The Brouwer degree has several properties, one of which is related to the Jacobian of  mappings:
 \begin{theo}\label{degree}
The mapping $F(x)$ satisfies  the following two conditions 

 (i) The mapping $F(x)$ is continuous in the closure of a bounded open set $D\subset\mathbb{R}^{n} $, and the equation $F(x)=z~(x\in\bar{D}, z \text{ fixed}) $
has finite number of distinct solutions $x_{1},\cdots, x_{N}$ which belongs to $D$. 

(ii) The mapping $F(x)$ is class $C^{1}$ in a vicinity of each point $x_{i} (i=1,\cdots, N)$ and the determinant of Jacobian $\det JF(x)$ doesn't vanish 
for $x=x_{i} (i=1,\cdots, N) $.
Furthermore, $N^{+}(resp. N^{-})$ is the number of points of set $x_{1},\cdots,x_{N}$ with positive Jacobian (resp. negative Jacobian).

Then Brouwer degree satisfies $$d[F(x); D, z]=N^{+}-N^{-} .$$

\end{theo}
\begin{rem}\label{const}
Consider a polynomial mapping $F$ from $ \mathbb{R}^{n}$ to $\mathbb{R}^{n}$ and let $q\in Im(F)$ with finite number
of points $p_{1},\cdots, p_{k}$ such that $\{p_{1},\cdots, p_{k}\}=F^{-1}(q)$, i.e., the pre-image of $q$ is finite number of points.  Moreover, suppose
the $\det JF$ is a non-zero constant and take  $D$, containing all $k $ points $p_{i}$,   to be a open bounded sub-set of $\mathbb{R}^{n}$  with $q\notin F(\partial D)$. Then, according to the Thm.\ref{degree}, $\sharp \{F^{-1}(q)\}=|d[F(x); D, q]|=k$, where the sign $|\cdot|$ represents the absolute value of real numbers. Since the $\det JF$ is a non-zero constant but can be positive or negetaive, so the absolute value $|~|$ will apear in the equality above.
\end{rem}
Now, we can formulate the homotopical invariance of the Brouwer degree (Thm. 3, \cite{Heinz}), which is the main property used in our proof.
\begin{theo}
Let $D$ be a bounded open set in $\mathbb{R}^{n}$ and $I$ be a closed interval $0\leq t \leq 1$. Furthermore, let the mapping be $y=F(x, t)$ be continuous for $(\bar{D}\times [0, 1])$, and $F(x,t)\neq z$ for $(x,t)\in (\partial D \times I)$, where $z$ is fixed in $\mathbb{R}^{n}$. Then $d[F(x,t);D,z]$
is a constant for $t\in I$.

\end{theo}
In the end of formualtion of the Brouwer degree, we give the explicit expression of homotopical invariance, which is essentailly needed in our proof. The following corollory have been taken as a definition in order to obtain the product formula of the Brouwer degree by Leray (Def.4, \cite{Heinz}).
\begin{coro}
Let the mappping $y=F(x)$ be continuous in the clousure of a bounded open set $D\subset \mathbb{R}^{n}$ .Furthermore, let $F(\partial D)$ be the image set of $\partial D$ under the tansformation $x\rightarrow F(x)$, and let $\Delta$ be a open connected component of the $\mathbb{R}^{n}\backslash F(\partial D)$.Then the Brouwer degree $d[F(x);D,z]$ is a constant for any $z\in \Delta$ 
\end{coro}
\section{Main results }
 For a polynomial map $F: \mathbb{C}^{n}\rightarrow \mathbb{C}^{p}$, the well-known result (see page 132 in \cite{Narasi} ) is that $F$ is open if and only if its fibers have pure dimension $n-p$. The result in the real case for $F$  from $\mathbb{R}^{n}$ to itself is the following theorem proved by J.Gamboa and F. Ronga in \cite{GamRon}.

 \begin{theo}\label{open}
 Let $F: \mathbb{R}^{n}\rightarrow \mathbb{R}^{n}$ be a polynomial mapping. Then $F$ is an open mapping if and  only if the fibers of $F$ are finite and the sign of $\det J(F)$ does not change (i.e. $\det J(F)(x)\geqslant 0$, for $ \forall x\in \mathbb{R}^{n}$ or $\det J(F)(x)\leqslant 0$, for $ \forall x\in \mathbb{R}^{n}).$
  \end{theo}
By Thm.\ref{open}, the polynomial mapping $F$ is an open mapping, when $F$ is provided with two conditions, finite fibres and no changed sign of Jacobian determiant. In fact, the polynomial mapping $F$ is an open mapping only if $\det JF(x)$ is non-zero everywhere, i.e. $\det J(F)(x)> 0$ (or $\det J(F)(x)< 0$) which guarantees the finiteness of fibres. On the finiteness of fibers of $F$, this is already proved by M. Dru$\dot{z}$kowski and K. Tutaj (see Lemma 3.1, \cite{DruTu}). Actually, their results are more than finiteness. We directly cite it as my proposition without proof.
\begin{pro}\label{iso}
If $F: \mathbb{R}^{n}\rightarrow \mathbb{R}^{n}$ is a polynomial mapping such that $\det J(F)(x)\neq 0$ for every $x\in \mathbb{R}^{n}$. Then for every $a\in \mathbb{R}^{n}$ , the equation $F(x)=a$ has only isolated solutions and 
$$\sharp\{x\in \mathbb{R}^{n}:  F(x)=a\}\leqslant (\text{deg}F_{1})\cdot\cdots(\text{deg}F_{n}), $$ where we denote  the degree of every component $F_{i}$ in $n$-varibles by $ \text{deg}F_{i}$.
\end{pro}
\begin{rem}
It is very interesting to note that A. Fernandes, C. Maquera and J. Venato-Santos, gave a more general result (see Cor. 2.5 \cite{FerMaqVen} )by introducing semi-algebraic set and semi-algebraic map. Their result is : If $F: \mathbb{R}^{n}\rightarrow \mathbb{R}^{n}$ is semi-algebraic local homeomorphism, then there exists $k\in \mathbb{N}$ such that the cardinality of the fibers of $F$: $\sharp \{a \in \mathbb{R}^{n}|F(a)=p\}\leqslant k$ for all $p\in \mathbb{R}^{n}$.  The semi-algebraic mapping is somehow general than polynomial mappings and the local homeomorphism condition is corresponding to the Jacobian condition. In the paper (loc. cit.), they consider the Jacobian conjecture from the angle of topology by using foliation and semi-algebraic knowledge, which is worthy to reading. 
\end{rem}
The reduction of the degree of ploynomial mappings in (GJC) is the entry point of our proof and also is one of breakthrough points. In the following proposition, the conclusion is true for  $\mathbb{C}$, and $\mathbb{R}$. 
\begin{pro}(\cite{BCW},\cite{DruTu1})
It is sufficient to  consider in (GJC), for every $n\geq 2$, only polynomial mappings of the so-called cubic homogeneous form $F=I+H=(x_{1}+H_{1},\cdots, x_{n}+H_{n})$,
where $I$ denotes the identity, $H=(H_{1},\cdots ,H_{n})$ and $H_{i}:\mathbb{K}^{n}\rightarrow \mathbb{K}$ is a cubic homogeneous polynomial of degree $3$ or $H_{i}=0$  $i=1,\cdots ,n.$ 
\end{pro}
In \cite{Dru2} , for (GJC), the polynomial of cubic homogeneous form can be improved to be a better form (it is called the cubic linear form). In our proof, the cubic 
homegeneous form is good enough to obtain the conclusion. From the reduction of the degree, we can derive that any polynomial $F$ of cubic homegeneous form is globally injective on the origin of $Im(F)$.
\begin{pro}\label{glo}
For every $n\geq 2$ ,the polynomial mapping of cubic homogeneous form $F=I+H : \mathbb{R}^{n}\rightarrow \mathbb{R}^{n}$ with a non-zero constant Jacobian is globally injective at $0$,i.e.,
$F^{-1}(\{0\})=\{0\}$ in $\mathbb{R}^{n}$.
\end{pro}
\begin{proof}
Suppose, on the contrary, there exists $0\neq a \in \mathbb{R}^{n}$ such that $F(a)=F(0)=0$.
Since $F$ is a polynomial of cubic homogeneous form, so $F$ can be written by $F(X)=F(x_{1},\cdots,x_{n})=F_{(1)}(X)+F_{(3)}(X)$, where $F_{(d)}(X)$ are components of homogeneous degree of $d~(d=1, 3)$. By introducing a real parameter $t$, consider the mapping $F(tX)$ and its derative for $t$.
On the one hand,
 \begin{align*}
     0=& F_{(1)}(a)+F_{(3)}(a)\\
      =& F_{(1)}(a)+3\frac{1}{(\sqrt{3})^{2}}F_{(3)}(a)~~~~~~~(t_{0}=1/\sqrt{3})\\
      =& \frac{\mathrm{d}}{\mathrm{d}t}(F_{(1)}(a)t+ F_{(3)}(a)t^{3})|_{t=t_{0}}\\
      =& \frac{\mathrm{d}}{\mathrm{d}t}(F(ta))|_{t=t_{0}}\\ 
      =& J(F)(t_{0}a)\cdot a
 \end{align*}
On the other hand, by the condition of Jacobian, $J(F)(t_{0}a)$  is invertible and $a\neq 0$, hence $J(F)(t_{0}a)\cdot a \neq 0$. This is a contradiction.
Therefore, the polynomial $F$ of cubic homegeneous form is injective globally at $0$.
\end{proof}

\begin{theo}\label{injct}
The Generalized real Jacobian conjecture and the Generalized complex Jacobian conjecture are true.
\end{theo}
\begin{proof}
As is known to us, the Generalize real Jacobian conjecture can imply the the Generalized complex Jacobian conjecture. Also, it is sufficient to prove 
the Generalized real Jacobian conjecture if for every $n\geq 2$, any polynomial mapping of cubic homegeneous form $F: \mathbb{R}^{n}\rightarrow\mathbb{R}^{n}$ with non-zero constant Jacobian is injective.

By Proposition \ref{glo}, any polynomial mapping of cubic homegeneous form $F$ with non-zero constant Jocobian is globally injective at $0\in \mathbb{R}^{n}$.
For any point $b\neq 0$ in the image of $F$, i.e. $b\in Im(F)$, we will show $F^{-1}(\{b\})=\{a\}$ for some $a\in \mathbb{R}^{n}$. By the Prop.
\ref{iso}, for $\forall~ b\in Im(F)$, there exist only finite number of solutions for $F(X)=b$. Hence, suppose $F^{-1}(\{b\})=\{a_{1},\cdots, a_{N}\}$ for some fixed integer $N$. Denote the maximum of the norm of $N$ points by $r=\text{max}_{1\leq i\leq N}\{|a_{i}|\}$. For any selected $R>r$ (for instance,  $R=r+1$), consider the 
open bounded ball $B(0,R)=\{X\in \mathbb{R}^{n}| |X|<R\}$ centered at $0$. By Them.\ref{degree}, the Brouwer degree
 $d[f;B(0,R),b]=\sum_{X\in F^{-1}(b)}\text{Sign}JF(X)=\sharp \{F^{-1}(b)\}=N$, because of $b\notin F(\partial B(0,R))$ by the selection of $R$ and the non-zero constant Jacobian of $F$, where it is no harm to assume the $\det JF(X)$ is a positive constant by Rem.\ref{const} . 

Since $F$ is a continuous mapping and the open bounded ball $B(0,R)$ is a connected set in $\mathbb{R}^{n}$, $F(B(0,R))$ is a connected set again (theorem 4.22 in Principles of mathematical analysis, W.Ludin). Furtheremore, $0\in F(B(0,R))$ and $b\in F(B(0,R))$. Also, $0$ and $b$ are in the same open connected
set $F(B(0,R))$ by Thm.\ref{open}. By the homotopy invariance of the Brouwer defree i.e. , $d_{B}[f;B(0,R),b]=d_{B}[f;B(0,R),0]=1$ because $F$ is globally injective at $0$. This demonstrates the injectivity of $F$ by the arbitrariness of $b$.
\end{proof}

Roughly speaking, for a polynomial mapping $F:\mathbb{R}^{n}\rightarrow \mathbb{R}^{n}$ with $\det JF(X)=\text{constant}\neq 0$, the proof of the injectivity of $F$ have three steps: 

(1) Find a point $b$ from the image set of $F$ such that $F$ is globally injective on this point $b$, i.e. $F^{-1}(\{b\})=\{a\}$ for some $a \in\mathbb{R}^{n}$. In our case, the origin is a perfect candidate because $F$ is a polynomial mapping of cubic homegeneous form. 

(2) Find an open bounded connected sub-set $D$ such that $D$ contains $a$ ( in (1)) and finite number of points (all  pre-image points of  any point $q$ from the image set of $F$ ) and $q\notin F(\partial D)$. Since the polynomial mapping with a non-zero constant Jacobian guarantees that the mapping is a mapping of finte to one. So the open bounded connected sub-set $D$ exists. In our case, $D$ is taken as the open bounded ball $B(0, R)$.

(3)By the homotopy invariance of the Brouwer degree, $d_{B}[f;D,b]=d_{B}[f;D,q]=1$ , where $b$ and $q$, $D$, come from (1) and (2), respectively.

This shows that $F$ is injective.

Acoording the stratege of our proof, to a specific integer $n$, to consider the $(JC)_{n}$ and strong real Jacobian conjecture, we have the following  theorem:
\begin{theo}\label{srp}
The real Jacobian conjecture (and the strong real Jacobian conjecture) is  ture if and only if there exists a point $b\in Im(F)$ such that $F$ is globally injective on $b$, i.e., $\sharp \{a\in \mathbb{R}^{n}| F(a)=b\}=1$.
\end{theo}
\begin{rem}
By the Pinchuk's counter-example (\cite{Pin}), we know that the strong real Jacobian conjecture is false. The Thm.\ref{srp} shows that for a polynomial mapping $F$, if the $\det JF(x)$ is everywhere non-zero and $F$ is globally injective on one point in the image set of $F$, then the strong real Jacobian conjecture is true. It seems that the non-zero determiant of Jacobian of $F$ and the global injetivity on one point in the image set of $F$,  will imply the non-zero constant determiant of Jacobian of $F$, i.e., $\det JF(x)=\text{const}\neq 0$. Of course, similar to the conter-example of Pinchuk,  it is also non-trivial that how to find a point in the image set of $F$ such that F is globally injective on the point. 
\end{rem}
\section*{Acknowledgements}
I want to express Prof. A. van den Essen, who had done numerous work and organize to publish papers and the Monograph for this conjecture. Their work makes possible for young researcher entering this topic. The paper is from the extension of  my research in Yau Mathematical Center of Sciences, Tsinghua University, which provide a  free and relaxed environment for research. Also, I am indebted to the College of Sciences,China jiliang University, where I am supported a lot as a young scholar.

Address of authors:\\
 Quan XU
\\ College of Sciences, China Jiliang University, \\
310018, Qiantang District, Zhejiang Province, P. R. China.\\
Email: xuquan.math@gmail.com
\end{document}